\documentclass[12pt, reqno]{amsart}
\usepackage{amsmath, amsthm, amscd, amsfonts, amssymb, graphicx, color}
\usepackage[bookmarksnumbered, colorlinks, plainpages]{hyperref}
\usepackage{cite}
\newtheorem{thm}{Theorem}[section]

\newtheorem{rem}{Remark}[section]
\newtheorem{lem}{Lemma}[section]

\newcommand{\be}{\begin{equation}}
	\newcommand{\ee}{\end{equation}}
\newcommand{\beas}{\begin{eqnarray*}}
	\newcommand{\eeas}{\end{eqnarray*}}
\newcommand{\bea}{\begin{eqnarray}}
	\newcommand{\eea}{\end{eqnarray}}
\numberwithin{equation}{section}

\begin{document}
	\setcounter{page}{1}
	
\title[Bohr-Rogosinski inequalities involving Schwarz functions]{Bohr-Rogosinski inequalities involving Schwarz functions}
	
	\author[S. Ahammed and M. B. AHamed]{Sabir Ahammed and Molla Basir Ahamed$ ^* $}
	
	\address{Molla Basir Ahamed, Department of Mathematics, Jadavpur University, Kolkata-700032, West Bengal, India}
	\email{mbahamed.math@jadavpuruniversity.in}
	
		\address{Sabir Ahammed, Department of Mathematics, Jadavpur University, Kolkata-700032, West Bengal, India
		}
	\email{sabira.math.rs@jadavpuruniversity.in}
	
	
	\date{09-12-2023, File Name: Bohr-Oper-AA-P1.
		\newline\indent $^{1}$ Corresponding author}
	
\begin{abstract} 
In this paper, we study Bohr's inequality and refined versions of Bohr-Rogosinski inequalities involving Schwarz functions. Moreover, we establish a version of multidimensional analogue of Bohr inequality and Bohr-Rogosinski inequalities involving Schwarz functions. Finally, we establish a multidimensional analogue of the refined version of Bohr inequalities with the initial coefficient being zero. All the results are proved to be sharp. 
\end{abstract} \maketitle
	
\section{Introduction}
Let $ \mathcal{H}_{\infty} $ denote the space of all bounded analytic functions $f$ in the unit disk $ \mathbb{D}:=\{z\in\mathbb{C} : |z|<1\} $
equipped with the topology of uniform convergence on compact subsets of $ \mathbb{D} $ with the supremum norm $||f||_{\infty}:=\sup_{z\in \mathbb{D}}|f(z)|$. We define a subclass $ \mathcal{B} $ of $ \mathcal{H}_{\infty} $ as follows 
\begin{align*}
	\mathcal{B}:=\{f\in\mathcal{H}_{\infty} : f(z)=\sum_{n=0}^{\infty}a_nz^n\; \mbox{with}\; ||f||_{\infty}\leq1\}.
\end{align*} The improved version of the classical Bohr's inequality \cite{Bohr-1914} states that if $ f\in\mathcal{B} $, then the associated majorant series $ M_f(r):=\sum_{n=0}^{\infty}|a_n|r^n\leq 1 $ holds for $ |z|=r\leq 1/3 $ and the constant $ 1/3 $ cannot be improved. The constant $ 1/3 $ is famously known as the Bohr radius and the inequality $ M_f(r)\leq 1 $ is known as the Bohr inequality for the class $ \mathcal{B} $. \vspace{1.2mm}

Generalization of the classical Bohr inequality is now an active area of research and is extensively studied in recent years for different classes of functions. In the literature, Bohr’s power series theorem has been studied in many different situations what is called Bohr phenomenon. In the study of Bohr phenomenon, the initial term $ |a_0| $ in the majorant series $ M_f(r) $ plays a crucial role. For instance, Bohr's theorem has been studied by introducing the term $ |a_0|^p $, where $ 0<p\leq 2 $, instead of $ |a_0| $, with the corresponding best possible radius $ p/(p+2) $ (see \cite{Liu-Ponnusamy-PAMS-2021,Ponnusamy-JMAA-2022} and references therein). In certain cases, if  $|a_0|$ is replaced by $|f(z)|$, then the constant $1/3$ could be replaced by $\sqrt{5}-2$ which is best possible (see  \cite{Alkhaleefah-Kayumov-Ponnusamy-PAMS-2019}).  Detailed account of research on Bohr radius problem can be found in the survey article \cite{Ponnusmy-Survey} and references therein. Actually, Bohr’s theorem received greater interest after it was used by Dixon \cite{Dixon & BLMS & 1995} to characterize Banach algebras that satisfy von Neumann inequality. The generalization of Bohr’s inequality with different suitable functional settings is now an active area of research in function theory: for instance, Aizenberg \textit{et al.} \cite{aizenberg-2001}, and Aytuna and Djakov \cite{Aytuna-Djakov-BLMS-2013} have studied the Bohr property of bases for holomorphic functions; Ali \textit{et al.} \cite{Ali-Abdul-NG-CVEE-2016} have found the Bohr radius for the class of starlike log-harmonic mappings; while Paulsen \textit{et al.} \cite{Paulsen-PLMS-2002} extended the Bohr inequality to Banach algebras; Hamada \emph{et al.}\cite{Hamada-IJM-2009} have studied the Bohr's theorem for holomorphic mappings with values in homogeneous balls; Galicer \emph{et al.}\cite{Galicer-Mansilla-Muro-TAMS-2020} have studied mixed Bohr radius in several complex variables, and many authors studied Bohr-type inequalities for different class of functions (see \cite{Alkhaleefah-Kayumov-Ponnusamy-PAMS-2019,Allu-Arora-JMAA-2022,Allu-CMB-2022,Das-JMAA-2022,Lata-Singh-PAMS-2022,S. Kumar-PAMS-2022,Kumar-JMAA-2023} and references therein.)  However, it can be noted that not every class of functions has the Bohr phenomenon, for example, B\'an\'at\'aau \emph{et al.} \cite{Beneteau-2004} showed that there is no Bohr phenomenon in the Hardy space $ H^p(\mathbb{D},X), $ where $p\in [1,\infty).$ The concept of the Bohr radius of a pair of operators is introduced recently and in terms of the convolution function, a general formula for calculating the Bohr radius of the
Hadamard convolution type operator with a fixed initial coefficient is obtained in \cite{Khasyanov-JMAA-2024}. \vspace{1.2mm}

Schwarz function has numerous applications in the theory of complex functions. The Schwarz function of a curve in the complex plane $ \mathbb{C} $ is an analytic function that maps points of the curve to their complex conjugates. It can be used to generalize the \textit{Schwarz reflection principle} to reflection across arbitrary analytic curves, not just across the real axis. The Schwarz function exists for analytic curves. More precisely, for every non-singular, analytic Jordan arc $ \Gamma $  in $ \mathbb{C} $, there is an open neighborhood $ \Omega $  of $ \Gamma $  and a unique analytic function $ \widehat{\omega} $ on $ \Omega $  such that $ \widehat{\omega}(z)=\bar{z} $ for every $ z\in\Gamma $. In this paper, our aim is to establish sharp multidimensional Bohr-Rogosinski inequalities involving Schwarz functions. \vspace{1.2mm}

Similar to the Bohr radius, there is also a concept of Rogosinski radius. According to Rogosinski (see \cite{Rogosinski-1923} for details), the Rogosinski radius is defined as follows: if $f(z)=\sum_{n=0}^{\infty} a_{n}z^{n}\in\mathcal{B},$ and its corresponding partial sum of $f$ is defined by $S_{N}(z):=\sum_{n=0}^{N-1} a_{n}z^{n}$, then for every $N \in\mathbb{N}$,	$|S_{N}(z)|\leq1\; \mbox{for}\; |z|\leq {1}/{2}$. The radius $1/2$ is sharp and is known as the Rogosinski radius (see e.g. \cite{Schur-Szego-1925}). With the idea of the Rogosinski radius, Kayumov \emph{et al.} \cite{Kayumov-Khammatova-Ponnusamy-JMAA-2021} have considered a relevant quantity known as the Bohr-Rogosinski sum $R_{N}^{f}(z)$ which is defined by 
\begin{equation}\label{e-1.1}
	R_{N}^{f}(z):=|f(z)|+ \sum_{n=N}^{\infty} |a_{n}||z|^{n}.
\end{equation}
Here, the Bohr-Rogosinski radius is the largest number $ r_N>0 $ such that  $ R_{N}^{f}(z)\leq 1 $ for $ |z|\leq  r_N $, where $r_N$ is the positive root of the equation $2(1+r)r^N-(1-r)^2=0$.
The Bohr-Rogosinski inequality \eqref{e-1.1} was also generalized by replacing the Taylor coefficients $a_n$ partially or completely by higher order derivatives in \cite{Liu-Shang-Xu-JIA-2018}. Moreover, Liu \cite{Gang Liu-JMAA-2021} has generalized the Bohr-Rogosinski inequality by replacing the Taylor coefficient $a_n$ of $f$ by $f^n(\widehat{\omega}_n(z))/n!$ and $r^m$ by $|\widehat{\omega}_m(z)|$ in part or in whole, where both $\widehat{\omega}_n$ and $\widehat{\omega}_m$ are some Schwarz functions. For background information related to Bohr-type inequality involving Schwarz functions, we refer to the recent papers \cite{Allu-Arora-JMAA-2022,Huang-Liu-Ponnu-AMP-2020,Hu-Wang-Long-BMMSS-2021} and references therein.\vspace{1.2mm}

Refining Bohr inequality or Bohr-Rogosinski inequality for certain class of functions showing them sharp are of independent interests to the researchers. Liu \emph{et al.} \cite{Liu-Liu-Ponnusamy-2021} and Ponnusamy \emph{et al.} \cite{Ponnusamy-Vijayak-Wirths-RM-2020} have established independently several refined version of Bohr-Rogosinski inequality in the case of bounded analytic functions on $ \mathbb{D} $. For certain class of harmonic mappings,  Bohr inequality, Bohr-Rogosinski inequality and their sharp refinements are extensively studied in \cite{Aha-CMFT-2022,Aha-CVEE-2022,Aha-Aha-CMFT-2023,Aha-Allu-BMMSS-2022,Aha-Allu-MSCN-2023,Aha-Allu-RMJ-2022,Ahamed-AMP-2021,Ahamed-CVEE-2021}. We now recall a couple of results that we want to establish in multidimensional settings involving Schwarz functions. Henceforth, for a function of the form  $f(z)=\sum_{n=0}^{\infty}a_nz^n\in\mathcal{B},$ $N\in\mathbb{N}$ and $p=1,2,$ we define the functional $ \mathcal{A}^t_{p,f}(r) $ by
\begin{align*}
	\mathcal{A}^t_{p,f}(r):=|f(z)|^p+\sum_{n=N}^{\infty}|a_n|r^n+sgn(t)\sum_{n=1}^{t}\dfrac{r^N}{1-r}+\left(\dfrac{1}{1+|a_0|}+\dfrac{r}{1-r}\right)\sum_{n=t+1}^{\infty}|a_n|^2r^{2n}.
\end{align*}
\begin{thm}\label{Thm-1.1}\cite{Liu-Liu-Ponnusamy-2021}
	For $N\in\mathbb{N},$ let $t=\lfloor (N-1)/2 \rfloor .$ If $f(z)=\sum_{n=1}^{\infty}a_nz^n\in \mathcal{B},$ then  $\mathcal{A}^t_{1,f}(r)\leq 1$ for $|z|=r\leq R_N,$ where $R_N$ is the unique positive root of equation $2(1+r)r^N-(1-r)^2=0.$  Moreover,
	$\mathcal{A}^t_{2,f}(r)\leq 1$
	for $r\leq R^{\prime}_N,$ where $ R^{\prime}_N$ is the unique positive root of equation $(1+r)r^N-(1-r)^2=0.$ The radii $R_N$ and $ R^{\prime}_N$ both are best possible.	
\end{thm}  
\begin{thm}\label{Thm-1.2}\cite{Ponnusamy-Vijayak-Wirths-RM-2020}
	If $f(z)=\sum_{n=1}^{\infty}a_nz^n\in \mathcal{B},$
	then 
	\begin{align*}
		\sum_{n=1}^{\infty}|a_n|r^n+\left( \dfrac{1}{1+|a_1|}+\dfrac{r}{1-r}\right)\sum_{n=2}^{\infty}|a_n|^2r^{2n-1}\leq 1 \;\;\mbox{for}\;\;  r\leq\dfrac{3}{5}.
	\end{align*}
	The number $3/5$ is best possible.
\end{thm}
There are some recent research findings on the Bohr phenomenon for the class $ \mathcal{B} $ involving Schwarz functions (see \cite{Allu-Arora-JMAA-2022,Huang-Liu-Ponnu-AMP-2020,Hu-Wang-Long-BMMSS-2021,Gang Liu-JMAA-2021}) but till date, there are no results on Bohr-Rogosinski inequalities in multidimensional settings involving Schwarz functions. In this paper, we take the opportunity to explore sharp Bohr-Rogosinski inequalities. Henceforth, throughout the discussion, 
for $ k\in\mathbb{N} $, we suppose that 
\begin{align*}
	\mathcal{B}_{k}&=\bigg\{\widehat{\omega}\in \mathcal{B} : \widehat{\omega}(0)=\widehat{\omega}^{\prime}(0)=\cdots=\widehat{\omega}^{(k-1)}(0)=0\;\mbox{and}\; \widehat{\omega}^{(k)}\neq 0\bigg\}.
\end{align*} 
The members of the class $\mathcal{B}_{k}$ are called the Schwarz functions. To obtain our desired results, we first prove several lemmas in details and proof of the main results will follow the argument of the lemmas. \vspace{1.2mm}

Before we proceed further, we need to introduce some concepts in $ \mathbb{C}^n $. Here and hereafter, we use the standard multi-index notation  $\alpha$ which is an $n$-tuple $(\alpha_1,\alpha_2,\dots \alpha_n)$ of non-negative integers, $|\alpha|$ is the sum $\alpha_1+\cdots+\alpha_n$ of its components, $\alpha!$ is the product $\alpha_1!\alpha_2!\dots \alpha_n!,$  $z$ denotes an $n$-tuple $(z_1,z_2,\dots z_n)$ of complex numbers, and $z^n$ denotes the product $z^{\alpha_1}_1z^{\alpha_2}_2\dots z^{\alpha_n}_n. $
Let $\mathbb{D}^n:=\{z\in\mathbb{C}^n:z=(z_1,z_2, \dots ,z_n), |z_j|<1,j=1,2,\dots,n\}$ be the open unit polydisk. Also, let ${\mathcal{K}_n}$ be the largest non negative number such that if the $n$- variables power series $\sum_{\alpha}a_{\alpha}z^{\alpha}$ 
converges in $\mathbb{D}^n$ and its sum $f$ has modulus less than $1$ so that $a_{\alpha}={\partial}^{\alpha}f(0)/\alpha!,$ then $ \sum_{\alpha}|a_\alpha||z^\alpha|<1\;\;\mbox{for all}\;\; z\in {\mathcal{K}_n}.\mathbb{D}^n. $
\vspace{1.2mm}

If $Q\subset{\mathbb{C}^n}$ be a complete circular domain  centered at $0\in Q ,$  then every holomorphic function $f$ in $Q$ can be expanded into sum of homogeneous polynomials given by 
\begin{align}\label{ee-1.3}
	f(z)=\sum_{n=0}^{\infty}P_n(z)\;\; \mbox{for}\;\; z\in Q,
\end{align}
where $ P_n(z)=\sum_{|\alpha|=n}A_{\alpha}z^{\alpha}$ is a homogeneous polynomial of degree $n,$ and $P_0(z)=f(0).$ In 1$ 997 $, Boas and Khavinson \cite{Boas-1997} introduced the $ N $-dimensional Bohr radius $ K_N, $ $ (N>1$ ) for the polydisk $ \mathbb{D}^N=\mathbb{D}\times\cdots\times\mathbb{D} $ which generates extensive research activity in Bohr phenomenon (see, e.g. \cite{Aizenberg-Djakov-PAMS-2000,Paulsen-PLMS-2002,Hamada-IJM-2009,Galicer-Mansilla-Muro-TAMS-2020,S. Kumar-PAMS-2022,Lin-Liu-Ponnusamy-Acta-2023}). In fact, Boas and Khavinson \cite{Boas-1997} proved that the $ N $-Bohr radius $ K_N $ as the largest radius $ r>0 $ such that for every complex polynomials $ \sum_{\alpha\in\mathbb{N}_0^N}c_{\alpha}z^{\alpha} $ in $ N $ variables 
\begin{align*}
	\sup_{z\in r\mathbb{D}^N} \sum_{\alpha\in\mathbb{N}_0^N}|c_{\alpha}z^{\alpha}|\leq \sup_{z\in \mathbb{D}^N} \bigg|\sum_{\alpha\in\mathbb{N}_0^N}c_{\alpha}z^{\alpha}\bigg|.
\end{align*}
As expected, the constant $ K_N $ is defined as the largest radius $ r $ satisfying $ \sum_{\alpha}|c_{\alpha}z^{\alpha}|<1 $ for all $ z $ with $ ||z||_{\infty}:=\max\{|z_1|, |z_2|, \ldots, |z_N|\}<r $ and all $ f(z)=\sum_{\alpha}c_{\alpha}z^{\alpha}\in\mathcal{H}(\mathbb{D}^N) $. For different aspects of multidimensional Bohr phenomenon including recent advances, we refer to the articles by  Aizenberg \cite{Aizn-PAMS-2000}, Liu and Ponnusamy \cite{Liu-Ponnusamy-PAMS-2021}, Paulsen \cite{Paulsen-PLMS-2002}, Defant and Frerick \cite{Defant-Frerick-IJM-2001}, Kumar\cite{S. Kumar-PAMS-2022} and (also see \cite{Liu-Liu-JMAA-2020,Lin-Liu-Ponnusamy-Acta-2023} and references therein).\vspace{1.2mm}

We recall here some of the results discussed above.
\begin{thm}\label{Thm-1.3}\cite{Aizn-PAMS-2000,Liu-Ponnusamy-PAMS-2021}
	If the series \eqref{ee-1.3} converges in the domain $Q$ and the estimate $|f(z)<1 |$ holds in $Q,$ then
	\begin{enumerate}
		\item[(a)] $\sum_{n=0}^{\infty}|P_n(z)|<1$ in the homothetic domain $(1/3)Q.$ Moreover, if $Q$ is convex, then $1/3$ is the best possible constant.\vspace{2mm}
		\item[(b)] for $p\in (0,2],$ we have 
		\begin{align*}
			|f(z)|^p+\sum_{n=N}^{\infty}|P_n(z)|\leq 1
		\end{align*} 
		in the homothetic domain $(R_{N,p})Q,$ where $R_{N,p}$ is positive root of the equation $2(1+r)r^N-p(1-r)^2=0.$ If $Q$ is convex, then the radius $R_{N,p}$ is best possible. \vspace{2mm}
		\item[(c)] for $p>0$, 
		\begin{align*}
			|f(z)|^p+\sum_{n=1}^{\infty}|P_n(z)|+\left(\frac{1}{1+a}+\frac{r}{1-r}\right)\sum_{n=1}^{\infty}|P_n(z)|^2\leq 1
		\end{align*} 
		for $z$ in the homothetic domain $(r_{a,p}).Q$ and $r\leq r_{a,p}, $ where $r_{a,p}$ is the minimum  positive root in $(0,1) $ of the equation $[1-(2-a^2)r](1+ar)^p-(1-r)(r+a)^p=0$ and $a=|f(0)|$ . Moreover, if $Q$ is convex, the number $r_{a,p}$ cannot be improved.\vspace{2mm}
		\item[(d)] if $Q$ is a complete circular domain center at $0\in Q\subset \mathbb{C}^n$ and $f(0)=0.$ Then $\sum_{n=1}^{\infty}|P_n(z)|\leq 1$ in the homothetic domain $\frac{1}{\sqrt{2}}Q.$ Moreover, if $Q$ is convex, then the number $1/\sqrt{2}$ cannot be improved.
	\end{enumerate}	
\end{thm}
\section{Main Results}
We obtain the following result as the multidimensional analogue of the Bohr inequality involving Schwarz functions.
\begin{thm} \label{th-2.1} If the series \eqref{ee-1.3} converges in the domain $Q,$  $a=|f(0)|<1,$ and the estimate $|f(z)|<1$ holds in $Q,$ then
	\begin{align}\label{ee-2.1}
		\sum_{n=0}^{\infty}|P_n(\widehat{\omega}_k(z))|\leq 1
	\end{align}
	in the homothetic domain $(1/\sqrt[k]{3})Q.$ Moreover, if $Q$ is convex, then the number $1/\sqrt[k]{3}$ is the best possible.
\end{thm}
\begin{rem}  It is worth noticing that
	\begin{enumerate}
		\item[(i)]in particular, when $P_n(z)=a_nz^n$ and $\widehat{\omega}_{1}(z)=z$ for $z\in\mathbb{D} \subset\mathbb{C}$, then we see that Theorem \ref{th-2.1} and the classical Bohr's theorem \cite{Bohr-1914} are identical.
		\item[(ii)] In particular, when $\widehat{\omega}_{k}(z)=z$ for $z\in Q \subset\mathbb{C}^n$, then we see that Theorem \ref{th-2.1} coincides with Theorem \ref{Thm-1.3}(a).
	\end{enumerate}
\end{rem}
Our next result is the following which is a generalization of Theorem \ref{Thm-1.3}(d). 
\begin{thm} \label{th-2.2}
	Suppose that $Q$ is a complete circular domain centered at $0\in Q\subset\mathbb{C}^n.$ If the series \eqref{ee-1.3} converges in the domain $Q,$  $a=|f(0)|<1$ and the estimate $|f(z)|<1$ holds in $Q,$ then
	\begin{align}\label{ee-2.2}
		\sum_{n=1}^{\infty}|P_n(\widehat{\omega}_k(z))|\leq 1
	\end{align}
	in the homothetic domain $(1/\sqrt[2k]{2})Q.$ Moreover, if $Q$ is convex, then the number $1/\sqrt[2k]{2}$ is best possible.
\end{thm}
\begin{rem} The following observations are clear.
	\begin{enumerate}
		\item[(i)] In particular, if $P_n(z)=a_nz^n$ and $\widehat{\omega}_{k}(z)=z$ for $z\in\mathbb{D} \subset\mathbb{C}$, then Theorem \ref{th-2.2} becomes identical with the result of Bombieri in \cite{Bombieri-1962}.
		\item[(ii)] In fact, if $\widehat{\omega}_{k}(z)=z$ for $z\in Q \subset\mathbb{C}^n$, then Theorem \ref{th-2.2} coincides with Theorem \ref{Thm-1.3}(d).
	\end{enumerate}
\end{rem}
To continue the study of Bohr-Rogosinski inequalities for functions of several complex variables, we obtain the next result as a generalization of Theorem \ref{Thm-1.3}(b) involving Schwarz functions.
\begin{thm} \label{th-2.3}
	If the series \eqref{ee-1.3} converges in the domain $Q,$ $p\in(0,2],$ $a=|f(0)|<1,$ and the estimate $|f(z)|<1$ holds in $Q,$ then
	\begin{align}\label{ee-2.3}
		|f(\widehat\omega_{m_0}(z))|^p+\sum_{n=N}^{\infty}|P_n(\widehat{\omega}_k(z))|\leq 1
	\end{align}
	in the homothetic domain $(R_{p,m_0}^{k,N})Q,$ 	where $ R^{k,N}_{p,m_0} $ is the minimum root in $ (0, 1) $ of equation $ Y(r)=0 $, where
	\begin{align}\label{ee-1.4}
		Y(r):=2r^{kN}\left(1+r^{m_0}\right)-p\left(1-r^{m_0}\right)\left(1-r^k\right).
	\end{align} Moreover, if $Q$ is convex, then  $R_{p,m_0}^{k,N}$ is best possible.
\end{thm}
\begin{rem} The following observations are clear.
	\begin{enumerate}
		\item In case of when $P_n(z)=a_nz^n$, $\widehat{\omega}_{k}(z)=z$ and $\widehat{\omega}_{m_0}(z)=z^m$ for $z\in\mathbb{D} \subset\mathbb{C}$, one can verify that Theorem \ref{th-2.3} and \cite[Theorems 2.1 and 2.2]{Kayumov-Khammatova-Ponnusamy-JMAA-2021} are identical for $p=1$ and $p=2$ respectively.
		\item  Suppose that $\widehat{\omega}_{m_0}(z) =\widehat{\omega}_{k}(z)=z$ for $z\in Q \subset\mathbb{C}^n.$ Then we see that Theorems \ref{th-2.3} and \ref{Thm-1.3}(b) are identical.
	\end{enumerate}
\end{rem}

We obtain the following result as a multidimensional analogue of Theorem \ref{Thm-1.1}  involving Schwarz functions. 
\begin{thm} \label{th-2.4}
	If the series \eqref{ee-1.3} converges in the domain $Q,$ $p\in (0,2],$ $a=|f(0)|<1,$ and the estimate $|f(z)|<1$ holds in $Q,$ then
	\begin{align}\label{ee-2.5}
		&\nonumber|f(\widehat{\omega}_{m_0}(z))|^p+\sum_{n=N}^{\infty}|P_n(\widehat{\omega}_k(z))|+sgn(t)\sum_{|\alpha|=1}^{t}|A_{\alpha}|^2\frac{|\widehat{\omega}_k(z)|^N}{1-|\widehat{\omega}_k(z)|}\\&\quad+\left(\frac{1}{1+|f(0)|}+\frac{|\widehat{\omega}_k(z)|}{1-|\widehat{\omega}_k(z)|}\right)\sum_{n=t+1}^{\infty}|P_n(\widehat{\omega}_k(z))|^2\leq 1
	\end{align}
	in the homothetic domain $(R^{k,N}_{p,m_0})Q,$ 	where $t=\left\lfloor (N-1)/2\right\rfloor$ for  $N\in \mathbb{N}$ and $ R^{k,N}_{p,m_0} $ is the minimum root in $ (0, 1) $ of equation \eqref{ee-1.4}.
	%\begin{align}\label{eee-2.6}
	%\mathcal{A}_{a,r}(r):=\left(1-(2-a^2)r^k\right)(1+ar^k)^p-(1-r^k)(r^k+a)^p=0.
	%\end{align}
	Moreover, if $Q$ is convex, then $R^{k,N}_{p,m_0}$ is best possible.
\end{thm}

\begin{table}[ht]
	\centering
	\begin{tabular}{|l|l|l|l|l|l|l|l|l|l|}
		\hline
		$\;\;k$& $\;2$&$\;2$& $\;3$& $ \;3 $&$\;3$& $\;7$& $ \;4 $& $\;5 $ \\
		\hline
		$m_0$& $\;2 $&$\;4 $& $\;3 $& $\;4$& $\;5 $& $\;3$& $\;5$ &$\; 7 $\\
		\hline
		$N$& $\;2$&$\;1 $& $\;2$& $\;1$&$\;2 $& $\;1$& $\;7$ &$\; 10 $\\
		\hline
		$p$& $\;0.12$&$\;0.6 $  &$ \;0.1 $& $\;1.2 $& $\;1.6 $&$\;2 $& $\;0.19$& $\;1.7$\\
		\hline
		$R^{k,N}_{p,m_0}$& $0.428676$&$0.463452 $& $0.54271$& $0.661436 $&$0.781955 $& $0.811851$& $0.861239$ &$ 0.940732 $\\
		\hline
	\end{tabular}\vspace{2.5mm}
	\caption{The approximate values of the root $ R^{k,N}_{p,m_0}$ for different values of $ k,m,N\in\mathbb{N} $ and $p\in(0,2]$ are shown in the table.}
\end{table}
\begin{figure}[!htb]
	\begin{center}
		\includegraphics[width=0.45\linewidth]{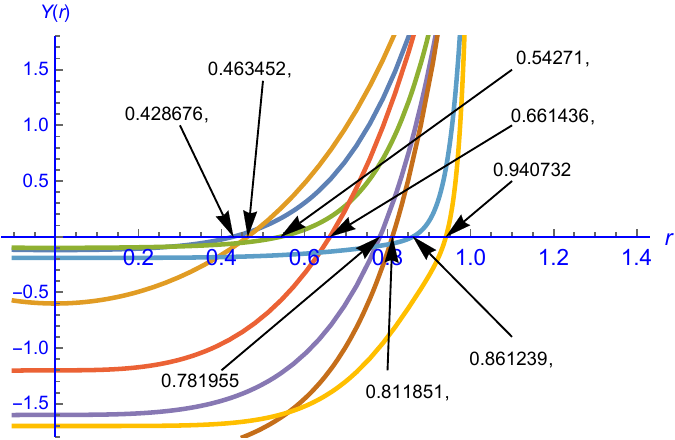}
	\end{center}
	\caption{The graph exhibits the roots $ R^{k,N}_{p,m_0} $ of Eq. \eqref{ee-1.4} for different values of $p, k, N$ and $m$ in the interval $ (0, 1) $.}
\end{figure}
\begin{rem} We note that
	\begin{enumerate}
		\item[(i)] if, in particular, $P_n(z)=a_nz^n$ and $\widehat{\omega}_{k}(z)= \widehat{\omega}_{m_0}(z)=z$ for $z\in\mathbb{D} \subset\mathbb{C},$ then Theorems \ref{th-2.4} and \ref{Thm-1.1} for $p=1,2$ are identical.
		\item[(ii)]  if, in particular, $\widehat{\omega}_{m_0}(z) =\widehat{\omega}_{k}(z)=z$ for $z\in Q \subset\mathbb{C}^n$ and $N=1$, then Theorem \ref{th-2.4} is multidimensional analogue of \cite[Lemma 3.3]{Liu-Ponnusamy-PAMS-2021}.
	\end{enumerate}
\end{rem}
Further, we establish a sharp refined version of Theorem \ref{th-2.2} in multidimensional analogue involving Schwarz functions.
\begin{thm} \label{th-2.5}
	If the series \eqref{ee-1.3} converges in the domain $Q,$  and the estimate $|f(z)|<1$ holds in $Q$ with $f(0)=0,$ then	
	\begin{align}\label{ee-2.6}
		&\sum_{n=1}^{\infty}|P_n(\widehat{\omega}_k(z))|+\left(\frac{|\widehat{\omega}_{k}(z)|^{-1}}{1+|A_{\alpha}|}+\frac{1}{1-|\widehat{\omega}_k(z)|}\right)\sum_{n=2}^{\infty}|P_n(\widehat{\omega}_k(z))|^2\leq 1
	\end{align}
	in the homothetic domain $(\sqrt[k]{3/5})Q,$ 	where  $|A_{\alpha}|=|A_{\alpha_1}|+\dots+|A_{\alpha_n}|$ such that $\alpha_1+\dots+\alpha_n=1$. Moreover, if $Q$ is convex, then $\sqrt[k]{3/5}$ is best possible.
\end{thm}
\begin{rem}
	It is easy to see that, in particular, if $ P_n(z)=a_nz^n $ and $\widehat{\omega}_{k}(z)=z,$ for $z\in \mathbb{D}\subset\mathbb{C},$ then Theorem \ref{Thm-1.2} becomes a particular form of Theorem \ref{th-2.5}.
\end{rem}
\section{Key lemmas and their proofs}
In this section, we state some key lemmas and their proofs which will be needed to prove the main results of this paper. In fact, these will play a key role in shortening the proof of the main results.
\begin{lem} \label{Lem-1.1}
	Suppose that $ f(z)=\sum_{n=0}^{\infty}a_nz^n\in\mathcal{B} $, and  $ \widehat{\omega}_{k}\in\mathcal{B}_{k}  $, $k\in\mathbb{N}$.  We have the sharp inequality:
	\begin{align*}
		\mathcal{M}^{\widehat{\omega}}_f(r):&=\sum_{n=0}^{\infty}|a_n||\widehat{\omega}_k(z)|^n\leq 1\; \mbox{for}\; |z|=r\leq \dfrac{1}{\sqrt[k]{3}}.
	\end{align*} 
	The number $1/\sqrt[k]{3}$ cannot be improved.
\end{lem}
\begin{proof}[\bf Proof of Lemma \ref{Lem-1.1}.]
	Since $ f(z)=\sum_{n=0}^{\infty}a_nz^n\in\mathcal{B} $, we must have $ |a_k|\leq 1-|a_0|^2 $ for integer $ n\geq 1 $(see \cite[P.35]{Graham-2003}).
	Let $ \widehat{\omega}_k\in\mathcal{B}_k $. By Scharz lemma, we have $ |\widehat{\omega}_k(z)|\leq |z|^k $ for all $ z\in\mathbb{D} $.  For $ z\in\mathbb{D} $ with $ |z|=r $, we see that 
	\begin{align*}
		\sum_{n=0}^{\infty}|a_n||\widehat{\omega}_k(z)|^n&\leq |a_0|+(1-|a_0|^2)\dfrac{r^k}{1-r^k}\\&=1+\left(1-|a_0|\right)\left(-1+(1+|a_0|)\frac{r^k}{1-r^k}\right).
	\end{align*}
	This gives that $ \sum_{n=0}^{\infty}|a_n||\widehat{\omega}_k(z)|^n\leq 1 $ for $ |z|=r\leq 1/\sqrt[k]{3} $. For sharpness, we consider the function $ f_a\in \mathcal{B} $ defined by 
	\begin{align}\label{eee-3.1}
		f_a(z):=\frac{a-z}{1-az}=a-\left(1-a^2\right)\sum_{n=1}^{\infty}a^{n-1}z^{n}=a_0+\sum_{n=1}^{\infty}a_nz^n,
	\end{align}
	where $ a_0=a $, $ a_n=-\left(1-a^2\right)a^{n-1} $ for some $a\in(0,1)$ and consider $ \widehat{\omega}_k(z)=z^k $. Taking $ |z|=r<1 $, we get the following estimate
	\begin{align*}
		\sum_{n=0}^{\infty}|a_n||\widehat{\omega}_k(z)|^n=a+\sum_{n=1}^{\infty}(1-a^2)a^{n-1}r^{kn}=1+\frac{(1-a)M(a, r)}{\left(1-ar^k\right)},
	\end{align*}
	where $ M(a, r):=-\left(1-ar^k\right)+\left(1+a\right)r^k $. It is easy to see that $ \sum_{n=0}^{\infty}|a_n||\widehat{\omega}_k(z)|^n>1 $ if, and only if, $ M(a, r)\geq 0 $ for $ r\geq 1/\sqrt[k]{3} $. In fact, we see that 
	\begin{align*}
		\lim\limits_{a\rightarrow 1^{-}}M(a, r)=-\left(1-r^k\right)+2r^k>0\; \mbox{for}\; r>\frac{1}{\sqrt[k]{3}}.
	\end{align*}
	This shows that the number $1/\sqrt[k]{3}$ is sharp and 
	the proof is complete . 
\end{proof}
We obtain the following lemma which generalizes that in \cite{Bombieri-1962}.
\begin{lem} \label{Lem-1.12}
	Suppose that $ f(z)=\sum_{n=0}^{\infty}a_nz^n\in\mathcal{B} $, and  $ \widehat{\omega}_{k}\in\mathcal{B}_{k}  $, $k\in\mathbb{N}$.  We have the sharp inequality:
	\begin{align*}
		\sum_{n=1}^{\infty}|a_n||\widehat{\omega}_k(z)|^n\leq 1\; \mbox{for}\; |z|=r\leq \dfrac{1}{\sqrt[2k]{2}}.
	\end{align*} 
	The number $1/\sqrt[2k]{2}$ cannot be improved.
\end{lem}
\begin{proof}[\bf Proof of Lemma \ref{Lem-1.12}.]
	Using the inequality $|\widehat{\omega}_k(z)|\leq |z|^k$ for $z\in\mathbb{D},$ where $\widehat{\omega}_k\in \mathcal{B}_k$ and $k\in \mathbb{N}$ and applying the similar proof of argument of Bombieri \cite{Bombieri-1962} (see also \cite{Liu-Ponnusamy-PAMS-2021}), the lemma can be proved easily. Hence we omit the details.
\end{proof}
The next lemma, we obtain as a generalization of (Theorem 2.1 \cite{Kayumov-Khammatova-Ponnusamy-JMAA-2021}) involving Schwarz functions.

\begin{lem}\label{BS-lem-1.2} Suppose that $ f(z)=\sum_{n=0}^{\infty}a_nz^n\in\mathcal{B} $, and $ \widehat{\omega}_{m_0}\in\mathcal{B}_{m_0} $, $ \widehat{\omega}_{k}\in\mathcal{B}_{k}  $, $m_0,k\in\mathbb{N}$. Then for $ p\in (0, 2] $ and $ N\in\mathbb{N} $, we have the sharp inequality:
	\begin{align*}
		\mathcal{I}^{\widehat{\omega}}_f(r):&=|f(\widehat{\omega}_{m_0}(z))|^p+\sum_{n=N}^{\infty}|a_n||\widehat{\omega}_k(z)|^n\leq 1\; \mbox{for}\; |z|=r\leq R^{k,N}_{p,m_0},
	\end{align*}
	where $ R^{k,N}_{p,m_0} $ is the unique root in $ (0, 1) $ of equation 
	\begin{align}\label{ee-1.1}
		2r^{kN}\left(1+r^{m_0}\right)-p\left(1-r^{m_0}\right)\left(1-r^k\right)=0.
	\end{align}
	The radius $ R^{k,N}_{p,m_0} $ is best possible.
\end{lem}

\begin{proof}[\bf Proof of Lemma \ref{BS-lem-1.2}]
	According to the assumption, $ f\in\mathcal{B} $, $ a:=|a_0|\in [0, 1) $ and $ \widehat{\omega}_k\in\mathcal{B}_k $ and $ \widehat{\omega}_{m_0}\in\mathcal{B}_{m_0} $. By the Schwarz lemma and Schwarz-Pick lemma, respectively, we obtain
	\begin{align*}
		|\widehat{\omega}_{m_0}(z)|\leq |z|^{m_0},\; |\widehat{\omega}_k(z)|\leq |z|^k\; \mbox{and}\; |f(z)|\leq \frac{|z|+a}{1+a|z|}\;\; \mbox{for}\;\; z\in\mathbb{D}.
	\end{align*}
	For $ \widehat{\omega}_k\in\mathcal{B}_k $ and $ \widehat{\omega}_{m_0}\in\mathcal{B}_{m_0} $, and as $ f\in\mathcal{B} $, we have $ |a_n|\leq 1-|a|^2 $ for $ n\geq 1 $ (see, \cite[p.35]{Graham-2003}). Thus it follows that 
	\begin{align*}
		|f\left(\widehat{\omega}_{m_0}(z)\right)|\leq \frac{|\widehat{\omega}_{m_0}(z)|+a}{1+a|\widehat{\omega}_{m_0}(z)|}\leq \frac{r^{m_0}+a}{1+ar^{m_0}}\; \mbox{for}\; |z|=r<1.
	\end{align*}
	A simple computation gives that
	\begin{align}\label{SB-eq-1.2}
		\mathcal{I}^{\widehat{\omega}}_f(r)&\leq \left(\frac{r^{m_0}+a}{1+ar^{m_0}}\right)^p+\sum_{n=N}^{\infty}|a_n|r^{kn}\leq \left(\frac{r^{m_0}+a}{1+ar^{m_0}}\right)^p+(1-a^2)\frac{r^{kN}}{1-r^k}\\&=1-\Phi_{p,m_0,k,N}(a),\nonumber
	\end{align}
	where
	\begin{align}\label{eee-3.4}
		\Phi_{p,m_0,k,N}(a):&=1-\left(\frac{r^{m_0}+a}{1+ar^{m_0}}\right)^p-(1-a^2)\frac{r^{kN}}{1-r^k}.
	\end{align}
	In order to establish the inequality $\mathcal{I}^{\widehat{\omega}}_f(r)\leq 1,$ it is sufficient to show that $ \Phi_{p,m_0,k,N}(a)\geq 0 $ for all $ a\in [0, 1] $. Note that $ \Phi_{p,m_0,k,N}(1)=0 $. We claim that $ \Phi_{p,m_0,k,N}(a) $ is a decreasing function of $ a $, under the condition of this theorem. A direct computation thus show that
	\begin{align*}
		\begin{cases}
			\Phi^{\prime}_{p,m_0,k,N}(a)=\displaystyle\frac{2ar^{kN}}{1-r^k}-p\left(1-r^{2m_0}\right)\frac{\left(r^{m_0}+a\right)^{p-1}}{\left(1+ar^{m_0}\right)^{p+1}},\vspace{2mm}\\
			\Phi^{\prime\prime}_{p,m_0,k,N}(a)=\displaystyle\frac{2r^{kN}}{1-r^k}-p\left(1-r^{2m_0}\right)\frac{\left(r^{m_0}+a\right)^{p-2}}{\left(1+ar^{m_0}\right)^{p+2}}\left(p-2-2ar^{m_0}-(p+1)r^{2m_0}\right).
		\end{cases}
	\end{align*}
	Evidently, $ \Phi^{\prime\prime}_{p,m_0,k,N}(a)\geq 0 $ for all $ a\in [0, 1] $, wherever $ 0<p\leq 1 $. Hence, for $ r\leq R^{k,N}_{p,m_0} $, we obtain that
	\begin{align*}
		\Phi^{\prime}_{p,m_0,k,N}(a)&\leq \frac{2r^{kN}\left(1+r^{m_0}\right)-p\left(1-r^{m_0}\right)\left(1-r^k\right)}{\left(1-r^k\right)\left(1+r^{m_0}\right)}\leq 0.
	\end{align*}
	Clearly, for each $ |z|=r\leq R^{k,N}_{p,m_0} $ and $ 0<p\leq 1 $, $ \Phi_{p,m_0,k,N}(a) $ is decreasing function of $ a\in [0, 1] $ which turns out that $ \Phi_{p,m_0,k,N}(a)\geq \Phi_{p,m_0,k,N}(1)=0 $ for all $ a\in [0, 1] $ and the desired inequality holds. Next, we show that condition $ \Phi^{\prime}_{p,m_0,k,N}(a)\leq 0 $ is also sufficient for the function $ \Phi_{p,m_0,k,N}(a) $ to be decreasing on $ [0, 1] $ in the case when $ 1<p\leq 2 $. To do this, we define a function 
	\begin{align*}
		\Psi_{p, m_0, a}(r):=\left(1+r^{m_0}\right)^2\frac{\left(r^{m_0}+a\right)^{p-1}}{\left(1+r^{m_0}a\right)^{p+1}}\;\; \mbox{for}\;\; r\in [0, 1).
	\end{align*}
	Our aim is to show that $ \Psi $ is an increasing function of $ r $ in $ [0, 1) $. An elementary computation yields that 
	\begin{align*}
		\Psi^{\prime}_{p, m_0, a}(r)=\left(1+r^{m_0}\right)\frac{\left(r^{m_0}+a\right)^{p-2}}{\left(1+ar^{m_0}\right)^{p+2}}H_{p, m_0, a}(r)\; \mbox{for} \; r\in [0, 1),
	\end{align*}
	where 
	\begin{align*}
		H_{p, m_0, a}(r):=(1-a)\big(r^{m_0}(1-a+p(1+a))+a(p+1)+p-1\big).
	\end{align*}
	Hence, for $ p>1 $ and $ a\in [0, 1] $, it follows that $ H(r)\geq 0 $ for all $ r\in [0, 1) $, and consequently, we have $ \Psi^{\prime}(r)\geq 0 $ for all $ r\in [0, 1) $. Thus we see that $ \Psi(r)\geq \Psi(0)=a^{p-1} $ for all $ r\in [0, 1) $ and for $ a\in [0, 1] $. Since $ 0\leq a^{2-p}\leq 1 $ for $ 1<p\leq 2 $, this observations helps to derive that for $ r\leq R^{k,N}_{p,m_0} $,
	\begin{align*}
		\Phi^{\prime}_{p,m_0,k,N}(a)&=\frac{2ar^{kN}}{1-r^k}-p\left(\frac{1-r^{m_0}}{1+r^{m_0}}\right)\Psi(r)\\&\leq a^{p-1}\left(\frac{2r^{kN}a^{2-p}}{1-r^k}-p\left(\frac{1-r^{m_0}}{1+r^{m_0}}\right)\right)\\&\leq a^{p-1}\left(\dfrac{2r^{kN}}{1-r^k}-p\left(\dfrac{1-r^{m_0}}{1+r^{m_0}}\right)\right)\\&=a^{p-1}\Phi^{\prime}_{p,m_0,k,N}(1)\\&\leq 0.
	\end{align*}
	Furthermore, $ \Phi_{p,m_0,k,N}(a) $ is a decreasing function of $ a\in [0, 1] $, wherever $ 1<p\leq 2 $ which implies that $ \Phi_{p,m_0,k,N}(a)\geq \Phi_{p,m_0,k,N}(1)=0 $ for all $a\in[0,1]$ and thus, the desired inequality holds. To show the radius $ R^{k,N}_{p,m_0} $ is sharp, we consider the function $f_a\in\mathcal{B}$ given by $\eqref{eee-3.1}$
	with $ \widehat{\omega}_{m_0}(z)=-r^{m_0} $, $ \widehat{\omega}_{k}(z)=z^k $. Taking $ |z|=r $, by a simple computation, we obtain
	\begin{align}\label{BS-eq-1.3}
		\mathcal{I}^{\widehat{\omega}}_{f_a}(r)=1+\frac{(1-a){Q}_{N,p,k}(a, r)}{(1-ar^k)(1+ar^{m_0})^p},
	\end{align}
	where
	\begin{align*}
		Q_{N,p,k}(a,r):&=(1-ar^{k})(1+ar^{m_0})^p\bigg(\left(\frac{1+a}{1-ar^k}\right)a^{N-1}r^{kN}\\&\quad-\frac{1}{(1-a)}\left(1-\left(\frac{r^{m_0}+a}{1+ar^{m_0}}\right)^p\right)\bigg).
	\end{align*}
	It is easy to see that the last expression on the right side of \eqref{BS-eq-1.3} is bigger than or equal to $ 1 $ if, and only if, $Q_{N,p,k}(a, r)\geq 0 $. In fact, for $ r>R^{k,N}_{p,m_0} $ and $ a $ close to $ 1 $, we see that 
	\begin{align*}
		\lim\limits_{a\rightarrow 1^{-}}Q_{N,p,k}(a, r)=(1+r^{m_0})^{p-1}\left(2r^{kN}(1+r^{m_0})-p(1-r^{m_0})(1-r^k)\right)>0,
	\end{align*}
	showing that the radius $ R^{k,N}_{p,m_0} $ is best possible. 
\end{proof}
Next, we establish the following lemma which is an improved version of that in \cite{Liu-Liu-Ponnusamy-2021} and it will help us to prove our result shortly.
\begin{lem}\label{BS-lem-1.1}
	Suppose that $ f(z)=\sum_{n=0}^{\infty}a_nz^n\in\mathcal{B} $, $ \widehat{\omega}_{m_0}\in\mathcal{B}_{m_0} $, $\widehat{\omega}_{k}\in\mathcal{B}_{k} $, $ m_0, k\in \mathbb{N} $. Then for $ p\in (0,2], $ we have the sharp inequality:
	\begin{align*}
		\mathcal{J}^{\widehat{\omega}}_f(r):&=|f(\widehat{\omega}_{m_0})|^p+\sum_{n=N}^{\infty}|a_n||\widehat{\omega}_k(z)|^n+sgn(t)\sum_{n=1}^{t}|a_n|^2\frac{|\widehat{\omega}_k(z)|^N}{1-|\widehat{\omega}_k(z)|}\\&\quad+\left(\frac{1}{1+|a_0|}+\frac{|\widehat{\omega}_k(z)|}{1-|\widehat{\omega}_k(z)|}\right)\sum_{n=t+1}^{\infty}|a_n|^2|\widehat{\omega}_k(z)|^{2n}\leq 1
	\end{align*}
	for $ |z|=r\leq R^{k,N}_{p,m_0} $, where $ R^{k,N}_{p,m_0} $ is the unique root in $ (0, 1) $ of equation  \eqref{ee-1.1} . The constant $ R^{k,N}_{p,m_0} $ cannot be improved. 
\end{lem}
The following lemma gives us the advantage to prove Lemma \ref{BS-lem-1.1}.  

\begin{lem}\cite{Liu-Liu-Ponnusamy-2021} \label{lem-3.5} 
	Suppose that $ f(z)=\sum_{n=0}^{\infty}a_nz^n\in\mathcal{B} $. Then for any N$\in\mathbb{N}$, the following inequality holds:
	\begin{align*}
		\sum_{n=N}^{\infty}|a_n|r^n&+sgn(t)\sum_{n=1}^{t}|a_n|^2\dfrac{r^N}{1-r}+\left(\dfrac{1}{1+|a_0|}+\dfrac{r}{1-r}\right)\sum_{n=t+1}^{\infty}|a_n|^2r^{2n}\\&\leq (1-|a_0|^2)\dfrac{r^N}{1-r}, 
	\end{align*}
	\;\;\mbox{for}\;\; $ r\in[0,1),$ where $t=\lfloor{(N-1)/2}\rfloor$.
\end{lem}
\begin{proof}[\bf Proof of Lemma \ref{BS-lem-1.1}]
	Firstly, we consider the first part of Lemma \ref{BS-lem-1.2}. By the Schwarz lemma and the Schwarz-Pick lemma, respectively, we obtain
	\begin{align*}
		|\widehat\omega_{m_0}(z)|\leq |z|^{m_0},\; |\widehat{\omega}_k(z)|\leq |z|^k\; \mbox{and}\; |f(\widehat\omega_{m_0}(z))|\leq \frac{r^{m_0}+a}{1+ar^{m_0}}\;\;\mbox{for}\;\;  z\in\mathbb{D}.
	\end{align*}
	In view of above inequalities and Lemma \ref{lem-3.5}, a simple computation shows that 
	\begin{align*}
		\mathcal{J}^{\widehat{\omega}}_f(r)&\leq \left(\frac{r^{m_0}+a}{1+ar^{m_0}}\right)^p+\sum_{n=N}^{\infty}|a_n|r^{kn}+sgn(t)\sum_{n=1}^{t}|a_n|^2\frac{r^{kN}}{1-r^k}\\&\quad+\left(\frac{1}{1+|a_0|}+\frac{r^k}{1-r^k}\right)\sum_{n=t+1}^{\infty}|a_n|^2r^{2kn}\\&\leq \left(\frac{r^{m_0}+a}{1+ar^{m_0}}\right)^p+(1-a^2)\frac{r^{kN}}{1-r^k}.
	\end{align*}
	The rest of the proof  follows from Lemma \ref{BS-lem-1.2} and hence the inequality $\mathcal{J}^{\widehat{\omega}}_f(r)\leq 1$  holds if, and only if, $r\leq R^{k,N}_{p,m_0}$, where $R^{k,N}_{p,m_0}$ is the minimum positive root in $(0,1)$ of the equation \eqref{ee-1.1}. For sharpness, we consider the function $ f_a $ is given by \eqref{eee-3.1} with $ \widehat{\omega}_{m_0}(z)=-r^{m_0} $, $ \widehat{\omega}_k(z)=z^k $. Taking $ |z|=r $, an easy computation gives that 
	\begin{align}\label{BS-eq-11.11}
		\mathcal{J}^{\widehat{\omega}}_{f_a}(r)&=\left(\frac{a+r^{m_0}}{1+ar^{m_0}}\right)^p+(1-a^2)\frac{r^{kN}a^{N-1}}{1-ar^k}+(1-a^2)(1-a^{2t})sgn(t)\\&\quad+\frac{(1+ar^k)(1-a^2)^2a^{2t}r^{2k(t+1)}}{(1+a)(1-r^k)(1-a^2r^{2k})}\nonumber\\&=1+\frac{(1-a)G^{N,k}_{p, m_0, t}(a,r)}{(1-ar^k)(1+ar^{m_0})^p},\nonumber
	\end{align} 
	where
	\begin{align*}
		&G^{N,k}_{p, m_0, t}(a,r)\\&:=(1-ar^k)(1+ar^{m_0})^p\bigg(\left(\frac{1+a}{1-ar^k}\right)a^{N-1}r^{kN}-\frac{1}{(1-a)}\left(1-\left(\frac{a+r^{m_0}}{1+ar^{m_0}}\right)^p\right)\bigg)\\&\quad+(1-ar^k)(1+ar^{m_0})^p\left((1+a)\left(1-a^{2t}\right)sgn(t)+\frac{\left(1-a^2\right)a^{2t}r^{2k(t+1)}}{\left(1-ar^k\right)\left(1-r^k\right)}\right).
	\end{align*}
	It is easy to see that the last expression on the right side of \eqref{BS-eq-11.11} is bigger than or equal to $ 1 $ if, and only if, $G^{N,k}_{p, m_0, t}(a,r)\geq 0 $. In fact, for $ r>R^{k,N}_{p,m_0}$, and $ a $ close to $ 1 $, we see that 
	\begin{align*}
		\lim\limits_{a\rightarrow 1^{-}} G^{N,k}_{p, m_0, t}(a,r)=(1-r^k)(1+r^{m_0})^p\left(\frac{2r^{kN}}{1-r^k}-p\left(\frac{1-r^{m_0}}{1+r^{m_0}}\right)\right)>0
	\end{align*}
	which implies that the constant $ R^{k,N}_{p,m_0}$ is best possible.
\end{proof}
The next lemma we establish involves Schwarz functions as a refined version of the Bohr inequality with the initial coefficient $|a_0|$ being zero. 
\begin{lem} \label{Lem-1.5}
	Suppose that $ f(z)=\sum_{n=1}^{\infty}a_nz^n\in\mathcal{B} $, and  $ \widehat{\omega}_{k}\in\mathcal{B}_{k},$ $k\in\mathbb{N}.$ We have the sharp inequality:
	\begin{align}\label{ee-3.5}
		\mathcal{L}^{\widehat{\omega}}_f(r):=\sum_{n=1}^{\infty}|a_n||\widehat{\omega}_k(z)|^n+\left(\frac{|\widehat{\omega}_k(z)|^{-1}}{1+|a_1|}+\frac{1}{1-|\widehat{\omega}_k(z)|}\right)\sum_{n=2}^{\infty}|a_n|^2|\widehat{\omega}_k(z)|^{2n}\leq1
	\end{align}	
	for $|z|=r\leq \sqrt[k]{3/5}.$ The number $\sqrt[k]{3/5}$ is cannot be improved.
\end{lem}
\begin{proof}[\bf Proof of Lemma \ref{Lem-1.5}]
	Suppose that $f(z)=\sum_{n=1}^{\infty}a_nz^n,$ where $|f(z)|\leq 1$  and $\widehat{\omega}_k\in \mathcal{B}_k$ for $z\in \mathbb{D}.$ We remark that the function $f$ can be represented as $f(z)=z\sum_{n=0}^{\infty}b_nz^n,$ where $b_n=a_{n+1},\; n\geq 0.$ By Schwarz lemma, we have $ |\widehat{\omega}_k(z)|\leq |z|^k $ for all $ z\in\mathbb{D} $. For the function $\sum_{n=0}^{\infty}b_nz^n\in \mathcal{B},$ $z\in\mathbb{D},$ using  Lemma \ref{lem-3.5} for $N=1,$ a simple computation shows that 
	\begin{align}\label{BS-eq-3.5}
		|b_0|&+\sum_{n=1}^{\infty}|b_n||\widehat{\omega}_k(z)|^n+\left(\frac{1}{1+|b_0|}+\frac{|\widehat{\omega}_k(z)|}{1-|\widehat{\omega}_k(z)|}\right)\sum_{n=1}^{\infty}|b_n|^2|\widehat{\omega}_k(z)|^{2n}\\&\leq \nonumber |b_0|+\left(1-|b_0|^2\right)\frac{|\widehat{\omega}_k(z)|}{1-|\widehat{\omega}_k(z)|}.
	\end{align}
	Multiplying both sides of \eqref{BS-eq-3.5} by $ |\widehat{\omega}_k(z)| $, we get
	\begin{align*}
		|b_0||\widehat{\omega}_k(z)|&+\sum_{n=1}^{\infty}|b_n||\widehat{\omega}_k(z)|^{n+1}+\left(\frac{1}{1+|b_0|}+\frac{|\widehat{\omega}_k(z)|}{1-|\widehat{\omega}_k(z)|}\right)\sum_{n=1}^{\infty}|b_n|^2|\widehat{\omega}_k(z)|^{2n+1}\\&\leq  |b_0||\widehat{\omega}_k(z)|+\left(1-|b_0|^2\right)\frac{|\widehat{\omega}_k(z)|^2}{1-|\widehat{\omega}_k(z)|}.
	\end{align*}
	Evidently, for $ |a_1|=|b_0|=a $ and $|z|=r$, we have 
	\begin{align}\label{BS-eq-3.6}
		&\sum_{n=1}^{\infty}|a_{n}||\widehat{\omega}_k(z)|^{n}+\left(\frac{1}{1+|a_1|}+\frac{|\widehat{\omega}_k(z)|}{1-|\widehat{\omega}_k(z)|}\right)\sum_{n=2}^{\infty}|a_{n}|^2|\widehat{\omega}_k(z)|^{2n-1}\\&\leq\nonumber ar^k+\left(1-a^2\right)\frac{r^{2k}}{1-r^k}\\&=1+F_k(a,r),\nonumber
	\end{align}
	where
	\begin{align*}
		F_k(a, r):=-1+ar^k+\left(1-a^2\right)\frac{r^{2k}}{1-r^k}.
	\end{align*}
	Differentiating w.r.t. `$ a $', we obtain
	\begin{align*}
		F^{\prime}_k(a, r)=r^k-\frac{2r^{2k}b}{1-r^k}\; \mbox{and}\; F^{\prime\prime}_k(a, r)=-\frac{2r^{2k}}{1-r^k}.
	\end{align*}
	It is easy to see that $ a=(1-r^k)/2r^k $ is critical point of $ F_k(a, r) $, and the maximum value of $F_k(a, r) $ is attained at $ a=(1-r^k)/2r^k $. Thus, an easy computation shows that
	\begin{align*}
		F_k(a, r)\leq F_k\left(\frac{1-r^k}{2r^k}, r\right)=\frac{\left(1+r^k\right)\left(5r^k-3\right)}{4\left(1-r^k\right)}\leq 0\; \mbox{for}\; r\leq \sqrt[k]{\frac{3}{5}}.
	\end{align*}
	Therefore, the desired inequality
	\begin{align*}
		\sum_{n=1}^{\infty}|a_n||\widehat{\omega}_k(z)|^n+\left(\dfrac{|\widehat{\omega}_k(z)|^{-1}}{1+|a_1|}+\frac{1}{1-|\widehat{\omega}_k(z)|}\right)\sum_{n=2}^{\infty}|a_n|^2|\widehat{\omega}_k(z)|^{2n}\leq1
	\end{align*}	
	holds for $|z|=r\leq \sqrt[k]{3/5}.$\vspace{1.2mm} 
	
	In order to show the constant $\sqrt[k]{3/5}$ is sharp, we consider the function $ f^*_a $ defined by 
	\begin{align}\label{ee-3.9}
		f^*_a(z)=z\left(\frac{a-z}{1-az}\right)=az-\left(1-a^2\right)\sum_{n=1}^{\infty}a^{n-1}z^{n+1}=a_1z+\sum_{n=2}^{\infty}a_nz^n,
	\end{align}
	where $ a_1=a $, $ a_n=-\left(1-a^2\right)a^{n-2}, $ $n\geq2$ and consider $ \widehat{\omega}_k(z)=z^k $.\vspace{1.2mm} 
	
	With this, taking $ |z|=r $, by a simple computation,  $\mathcal{L}^{\widehat{\omega}}_{f^*_a}(r)$ can be expressed as
	\begin{align*}
		&\sum_{n=1}^{\infty}|a_n||\widehat{\omega}_k(z)|^n+\left(\frac{1}{1+|a_1|}+\frac{|\widehat{\omega}_k(z)|}{1-|\widehat{\omega}_k(z)|}\right)\sum_{n=2}^{\infty}|a_n|^2|\widehat{\omega}_k(z)|^{2n-1}\\&=ar^k+\left(1-a^2\right)\frac{r^{2k}}{1-r^k}+\frac{\left(1+ar^k\right)\left(1-a^2\right)^2r^{3k}}{(1+a)\left(1-r^k\right)\left(1-a^2r^{2k}\right)}\\&=ar^k+\left(1-a^2\right)\frac{r^{2k}}{1-r^k}.
	\end{align*}
	Consequently, the claimed sharpness is shown by comparing this expression to the right-hand side of the expression in the formula $\eqref{BS-eq-3.6}$. Hence, we omit the details.
\end{proof}
\section{Proof of the main results}
\begin{proof}[\bf Proof of Theorem \ref{th-2.1}]
	To obtain the inequality $ \eqref{ee-2.1} $, we convert the multidimensional power series \eqref{ee-1.3} into the power series of one complex variable, and then we will use Lemma \ref{Lem-1.1}.  Let $ 	\Gamma=\{z=(z_1,z_2,\dots,z_n): z_j=b_jh, j=1,2,\dots,n;h\in\mathbb{D}\} $	be a complex line. Then, in each section of the domain $Q$ by the line $\Gamma$ the series \eqref{ee-1.3} covert into the following power series of complex variable $h:$
	\begin{align*}
		f(bh)=\sum_{n=0}^{\infty}P_n(b)h^n=f(0)+\sum_{n=1}^{\infty}P_n(b)h^n.
	\end{align*}
	Since $|f(bh)|<1$ for $h\in\mathbb{D}$ and $|f(0)|<1,$  by the Lemma \ref{Lem-1.1}, we obtain 
	\begin{align}\label{ee-3.13}
		\sum_{n=0}^{\infty}{|P_n(b)||\widehat{\omega}_{k}(h)|^n}\leq1 
	\end{align} 
	for $z$ in the section $\Gamma\cap(1/\sqrt[k]{3})Q.$ Since $\Gamma$ is an arbitrary complex line passing through the origin, the inequality \eqref{ee-3.13} is just same as the inequality \eqref{ee-2.1}.\vspace{1.2mm} 
	
	For the sharpness of the constant  $1/\sqrt[k]{3},$ let the domain $Q$ be convex. Then $Q$ is an intersection of half spaces 
	\begin{align*}
		Q=\bigcap_{b\in J}\{z=(z_1,z_2,\dots,z_n):Re(b_1z_1+\dots+b_nz_n)<1\}
	\end{align*}
 for some $J.$ Since $Q$ is a circular domain, we obtain $ 	Q=\bigcap_{b\in J}\{z=(z_1,z_2,\dots,z_n):|b_1z_1+\dots+b_nz_n|<1\}. $\vspace{1.2mm} 
 
 To show that the constant $1/\sqrt[k]{3}$ is the best possible, it is sufficient to show that $1/\sqrt[k]{3}$ cannot be improved for each domain $ 	R_b:=\{z=(z_1,z_2,\dots,z_n):|b_1z_1+\dots+b_nz_n|<1\}.$ In view of the Lemma \ref{Lem-1.1} for some $b\in[0,1),$  there exist a function $f_b:\mathbb{D}\rightarrow\mathbb{D}$ defined by \eqref{eee-3.1} with $|f_b(z)|<1$ for $z\in\mathbb{D}$  but for every $|z|=r>1/\sqrt[k]{3},$ it is easy to see that $\mathcal{M}^{\widehat{\omega}}_f(r)\leq1$ fails to hold in the disk $\mathbb{D}_r=\{z:|z|<r\}.$\vspace{1.2mm} 
 
 On the other hand, we consider the function $g:R_b \rightarrow \mathbb{D}$  defined by $g(z):=b_1z_1+\dots+b_nz_n.$ By performing an elementary computation, it can be shown that the function $f(z)=(f_b\;\circ\; g)(z^k)$ gives the sharpness of the constant $1/\sqrt[k]{3}$ for each domain  $R_b$. This completes the proof of the theorem. \vspace{1.2mm}
\end{proof}
\begin{proof}[\bf Proof of Theorem \ref{th-2.2}]
	By similar argument of the proof of Theorem \ref{th-2.1}, the result can be proof easily and hence, we omitted the details.
\end{proof}
\begin{proof}[\bf Proof of Theorem \ref{th-2.3}]
	In view of Lemma \ref{BS-lem-1.2} and the analogues proof of Theorem \ref{th-2.1}, the inequality \eqref{ee-2.3} can be obtain easily in the homothetic domain $(R_{p,m_0}^{k,N})Q,$ 	where $ R^{k,N}_{p,m_0} $ is the minimum root in $ (0, 1) $ of equation 
	\eqref{ee-1.4}. To prove the constant $R^{k,N}_{p,m_0}$ is best possible when $Q$ is convex, in view of the analogue proof of Theorem \ref{th-2.1}, it is enough to show that $R^{k,N}_{p,m_0}$ cannot be improved for each domain  $R_b.$ In fact, we see that the function  $f(z)=(f_b\;\circ\; g)(z^k)$ gives the sharpness of the constant $R^{k,N}_{p,m_0}$  in each domain $R_b.$ This completes proof of the theorem.   
\end{proof}	
\begin{proof}[\bf Proof of Theorem \ref{th-2.4}]
	The proof easily follows from Lemma \ref{BS-lem-1.1} and the analogue proof of Theorem \ref{th-2.3}. Hence, we omit the details. 
\end{proof}
\begin{proof}[\bf Proof of Theorem \ref{th-2.5}]
	By means of \eqref{ee-3.5} of Lemma \ref{Lem-1.5}, using the analogou proof of Theorem  \ref{th-2.1}, the inequality \eqref{ee-2.6} can be easily obtain  in the homothetic domain $(\sqrt[k]{3/5})Q.$ However,  prove the constant $\sqrt[k]{3/5}$ is best possible when $Q$ is convex, in view of the analogue proof of Theorem \ref{th-2.1}, it is enough to show that $\sqrt[k]{3/5}$ cannot be improved for each domain  $R_b.$ Here also, we see that the function  $f(z)=(f^*_b\;\circ\; g)(z^k)$ gives the sharpness of the constants $\sqrt[k]{3/5}$  in each domain $R_b,$ where the function $f^*_b:\mathbb{D}\rightarrow\mathbb{D}$ is given by \eqref{ee-3.9} for some $b\in (0,1)$. This completes proof of the theorem.   
\end{proof}
%\noindent{\bf Acknowledgment:} The authors would like to thank the anonymous referees for their elaborate comments and suggestions which will improve significantly the presentation of the paper. The second author is supported by UGC-JRF (NTA-Ref. No.: 201610135853), New Delhi, India.\vspace{1.5mm}
%\noindent\textbf{Compliance of Ethical Standards}

\noindent\textbf{Conflict of interest:} The authors declare that there is no conflict  of interest regarding the publication of this paper.\vspace{1.5mm}

\noindent\textbf{Data availability statement:}  Data sharing not applicable to this article as no datasets were generated or analysed during the current study.

\end{document}